\documentclass[reqno]{amsart}
\usepackage{hyperref}

\def\N{{\mathbb{N}}}

\def\R{{\mathbb{R}}}

\theoremstyle{plain}
\newtheorem{theorem}{Theorem}
\newtheorem{proposition}{Proposition}
\newtheorem{definition}{Definition}
\newtheorem{lemma}{Lemma} 

\theoremstyle{remark}

\newtheorem{example}{Examples}

\title[Probabilistic Arzela-Ascoli Theorem ]{Probabilistic Arzela-Ascoli Theorem }
\author{Mohammed Bachir, Bruno Nazaret}

\begin{document}

\date{05/03/2019} 
\address{Laboratoire SAMM 4543, Universit\'e Paris 1 Panth\'eon-Sorbonne\\
Centre P.M.F. 90 rue Tolbiac\\
75634 Paris cedex 13\\
France}

\email{Mohammed.Bachir@univ-paris1.fr}
\email{Bruno.Nazaret@univ-paris1.fr}
\begin{abstract}
We prove that, in the space of all probabilistic continuous functions from a probabilistic metric space $G$ to the set $\Delta^+$ of all cumulative distribution functions vanishing at $0$, the space of all $1$-Lipschitz functions is compact if and only if the space $G$ is compact. This gives a probabilistic Arzela-Ascoli type Theorem.

\end{abstract}
\maketitle
\vskip5mm
{\bf msc:} 54E70,  46S50.
\tableofcontents

\section{\bf Introduction}
The general concept of probabilistic metric spaces was introduced by K. Menger, who dealt with probabilistic geometry \cite{M1}, \cite{M2}, \cite{M3}. The decisive influence on the development of the theory of probabilistic metric spaces is due to B. Schweizer and A. Sklar and their coworkers in several papers \cite{SS0}, \cite{SS1}, \cite{SS2}, \cite{SS3}, see also \cite{SHE}, \cite{SHE1} \cite{KMP2}, \cite{HP} and \cite{HP1}. For more informations about this theory we refeer to the excellent monograph \cite{S.S}. 
\vskip5mm
Recently, the first author introduced in \cite{Ba} a natural concept of probabilistic Lipschitz maps defined from a probabilistic metric space $G$ into the set of all cumulative distribution functions that vanish at $0$, classically denoted by $\Delta^+$. In particular, the introduction of the space of all probabilistic $1$-Lipschitz maps provides a new method for the completion of probabilistic metric spaces extending a result of H. Sherwood in \cite{SHE1}. It also leads to a probabilistic version of the Banach-Stone theorem (see for instance \cite{Ba} and \cite{Ba2}).

\vskip5mm

In this paper, we investigate new properties of the space of all $1$-Lipschitz probabilistic maps defined on probabilistic compact metric spaces. More precisely, we give in our main Theorem \ref{AA} a probabilistic Arzela-Ascoli theorem, proving that the space of all $1$-Lipschitz probabilistic maps is a compact subset of the space of all probabilistic continuous functions equipped with the uniform modified L\'evy distance that we introduce in the next section.

\vskip5mm

This paper is organized as follows. In Section \ref{S1}, we recall classical results and notions about probabilistic metric spaces, triangle functions, the L\'evy distance and the weak convergence. In Section \ref{S2}, we recal the notion of probabilistic Lipschitz maps and probabilistic continuous functions introduced in \cite{Ba}. We also give some properties related to these notions. In Section \ref{S3}, we prove the Theorem \ref{AA}, which is our main result.

\section{\bf Classical Notions of Probabilistic Metric Spaces.} \label{S1}
In this section, we recall some general well know definitions and concepts about probabilistic metric spaces, as they can be found in \cite{HP}, \cite{S.S}, \cite{KMP2} and \cite{KMP3}. 
\vskip5mm
A (cumulative) distribution function is a function $F : [-\infty, +\infty] \longrightarrow [0, 1]$, nondecreasing
and left-continuous with $F(-\infty) = 0$; $F(+\infty) = 1$. The set of all distribution functions satisfying
$F(0) = 0$ will be denoted by $\Delta^+$. For $F, G \in \Delta^+$, the relation $F \leq G$ is understood as $F(t)\leq G(t)$, for  $t\in \R$. For all $a\in \R$, the distribution function $\mathcal{H}_a$ is defined by \[ \mathcal{H}_a(t)=
\left\{
\begin{array}{rl}
0& \textnormal{ if } t\leq a \\
1& \textnormal{ if } t>a,
\end{array}
\right.
\] 
and, for $a=+\infty$, by 
\[ \mathcal{H}_{\infty}(t)=
\left\{
\begin{array}{rl}
0& \textnormal{ if } t\in [-\infty, +\infty[ \\
1& \textnormal{ if } t=+\infty
\end{array}
\right.
\] 
It is well known that $(\Delta^+, \leq)$ is a complete lattice with respectively $\mathcal{H}_{\infty}$ and $\mathcal{H}_0$ as minimal and maximal element. Thus, for any nonempty set $I$ and any familly $(F_i)_{i\in I}$ of distributions in $\Delta^+$, the function $F=\sup_{i\in I} F_i$ is also an element of $\Delta^+$.
 
\subsection{Triangle function and  probabilistic metric space}
\begin{definition} \label{axiom1}  A binary operation $\tau$ on $\Delta^+$ is called a triangle function if and only if it is  commutative, associative, non-decreasing in each place, and has $\mathcal{H}_0$ as neutral element. In other words:
\begin{itemize}
\item[(i)] $\tau(F, L) \in \Delta^+$ for all $F, L \in \Delta^+$.
\item[(ii)] $\tau(F, L)=\tau(L, F)$ for all $F, L \in \Delta^+$.
\item[(iii)] $\tau(F,\tau(L, K))=\tau(\tau(F, L), K)$, for all $F,L,K\in \Delta^+$.
\item[(iv)] $\tau(F, \mathcal{H}_0)=F$ for all $F\in \Delta^+$.
\item[(v)] $F\leq L \Longrightarrow \tau(F, K) \leq \tau(L, K)$ for all $F, L, K\in \Delta^+$.
\end{itemize}
\end{definition}
\begin{definition} A triangle function $\tau$ is  said to be sup-continuous (see for instance \cite{C3}) if for all nonempty set $I$ and all familly $(F_i)_{i\in I}$ of distributions in $\Delta^+$ and all $L\in\Delta^+$, we have $$\sup_{i\in I} \tau(F_i, L)=\tau(\sup_{i\in I}(F_i),L).$$
\end{definition}
\vskip5mm
For simplicity of notations, in all what follows, the triangle function $\tau$ will be denoted by the binary operation $\star$ as follows:
$$\tau(L,K):=L\star K.$$
It follows from the axioms $(i)$-$(iv)$ that $(\Delta^+,\star)$ is an abelian monoid having $\mathcal{H}_0$ as a neutral element.
\begin{definition} \label{contin} Let $\star$ be a triangle function on $\Delta^+$.

$(1)$ A sequence $(F_n)$ of distributions in $\Delta^+$ converges weakly to a function $F$ in $\Delta^+$ if $(F_n(t))$ converges to $F(t)$ at each point $t$ of continuity of $F$. In this case, we write indifferently $F_n \,{\xrightarrow {\textnormal{w}}}\, F$ or $\lim_n F_n =F$.

$(2)$ We say that the law $\star$ is continuous at $(F,L)\in \Delta^+\times \Delta^+$ if we have $F_n\star L_n \,{\xrightarrow {\textnormal{w}}}\, F\star L$, whenever $F_n \,{\xrightarrow {\textnormal{w}}}\, F$ and $L_n \,{\xrightarrow {\textnormal{w}}}\, L$. 
\end{definition}

A classical class of triangle functions, that are both continuous and sup-continuous, is provided by operations taking the form : for all $F, L\in \Delta^+$ and for all $t\in \R$,
\begin{equation}\label{eq0}
(F\star_T L)(t):=\sup_{s+u=t} T(F(s),L(u)),
\end{equation}
where $T:[0,1]\times[0,1]\to[0,1]$, usually called a triangular norm (see  \cite{S.S,HP}), is left-continuous and satisfies
\begin{itemize}
\item $T(x,y) = T(y,x)$ ( commutativity);
\item $T(x,T(y,z)) = T(T(x,y),z)$ (associativity);
\item $T(x,y) \leq T(x,z)$ whenever $y\leq z$ (monotonicity );
\item $T(x, 1) = x$ (boundary condition).	
\end{itemize}

\begin{definition} \label{definition.M} Let $G$ be a set and let $D : G\times G \longrightarrow (\Delta^+,\star, \leq)$ be a map. We say that $(G, D, \star)$ is a probabilistic metric space if the following axioms hold:
\begin{itemize}
\item[(i)] $D(p,q)=\mathcal{H}_0$ iff $p=q$.
\item[(ii)] $D(p,q)=D(q,p)$ for all $p,q\in G$ 
\item[(iii)] $D(p,q)\star D(q,r)\leq D(p,r)$ for all $p,q, r\in G$ 
\end{itemize}
\end{definition}

This notion of probabilistic distance naturally leads to associated metric concepts, such as Cauchy sequence and completeness.

\begin{definition} In a a probabilistic metric space $(G, D, \star)$, a sequence $(z_n)\subset G$ is said to be a Cauchy sequence if for all $t\in \R$,
$$\lim_{n,p \longrightarrow +\infty} D(z_n, z_p)(t)=\mathcal{H}_0(t).$$
(Equivalently, if $D(z_n, z_p)\,{\xrightarrow {\textnormal{w}}}\,\mathcal{H}_0$, when $n, p\longrightarrow +\infty$). A probabilistic metric space $(G, D, \star)$ is said to be complete if every Cauchy sequence $(z_n)\subset G$ weakly converges to some $z_{\infty}\in G$, that is $\lim_{n \rightarrow +\infty} D(z_n, z_{\infty})(t)=\mathcal{H}_0(t)$ for all $t\in \R$, we will briefly note $\lim_n D(z_n, z_{\infty})=\mathcal{H}_0$. 
\end{definition}
\begin{example} \label{example.1}  Every complete metric space induce a probabilistic complete metric space. Indeed, if $d$ is a complete metric on $G$ and $\star$ is a triangle function on $\Delta^+$ satisfying $\mathcal{H}_a\star \mathcal{H}_b=\mathcal{H}_{a+b}$ for all $a, b\in \R^+$  (see the example in the formula \ref{eq0} below and references \cite{S.S} and \cite{HP}), then $(G, D, \star)$ where,
$$D(p,q)=\mathcal{H}_{d(p,q)}, \hspace{1mm} \forall p,q \in G, $$ 
is a probabilistic complete metric space. 
\end{example}

\subsection{ L\'evy distance, weak convergence and compactness} 
\begin{definition}
Let $F$ and $G$ be in $\Delta^+$ , let $h$ be in $(0, 1]$, and let $A(F, G; h)$
denote the condition
$$0 \leq G(t) \leq F(t + h) + h$$
for all $t \in (0, h^{-1})$.
The modified L\'evy distance is the map $d_L$ defined on $\Delta^+ \times \Delta^+$ as
$$d_L(F, G) = \inf \lbrace h : \textnormal{ both } A(F, G; h) \textnormal{ and } A(G, F; h) \textnormal{ hold} \rbrace.$$
\end{definition}
Note that for any $F$ and $G$ in $\Delta^+$, both $A(F, G; 1)$ and $A(G, F; 1)$ hold, whence $d_L$
is well-defined and $d_L(F, G) \leq 1$.
\vskip5mm
Recall from \cite{SS4} that the map $F\mapsto d_L(F,\mathcal{H}_0)$ is non-increassing, that is
\begin{eqnarray*}
F,G\in \Delta^+, F\leq G\Longrightarrow d_L(G,\mathcal{H}_0) \leq d_L(F,\mathcal{H}_0).
\end{eqnarray*}
We also recall the following results due to D. Sibley in \cite[Theorem 1. and Theorem 2]{DS}.
\begin{lemma} \label{DS1} (\cite{DS}) The function $d_L$ is a metric on $\Delta^+$ and $(\Delta^+,d_L)$ is compact.
\end{lemma}
\begin{lemma} \label{DS2} (\cite{DS}) Let $(F_n)$ be a sequence of functions in $\Delta^+$, and let $F$ be
an element of $\Delta^+$. Then $(F_n)$ converges weakly to $F$ if and only if $d_L(F_n, F)\longrightarrow 0$, when $n \longrightarrow +\infty$.
\end{lemma}
\begin{definition}
Let $(G,D,\star)$ be a probabilistic metric space. For $x \in G$ and $t > 0$, the strong $t$-neighborhood
of $x$ is the set
$$N_x(t) = \lbrace y \in G : D(x,y)(t) > 1 - t \rbrace,$$
and the strong neighborhood system for $G$ is $\lbrace N_x(t); x\in G, t > 0 \rbrace.$
\end{definition}

\begin{lemma} (\cite{S.S}) Let $t > 0$ and $x, y \in G$. Then we have
$ y\in N_x(t)$ if and only if $d_L(D(x,y), \mathcal{H}_0) < t$.
\end{lemma}
\begin{definition} A complete probabilistic metric space $(K,D,\star)$ is called compact if for all $t>0$, the
open cover $\lbrace N_x(t): x\in K \rbrace$ has a finite subcover.
\end{definition}
\begin{proposition} \textnormal{ (\cite[Theorem 2.2, Theorem 2.3]{MON})}  Let $(K,D,\star)$  be a complete probabilistic metric space. Then, we have:

$(1)$  $(K,D,\star)$  is compact if and only if every sequence has a convergent subsequence.

$(2)$ If $(K,D,\star)$ is compact, then  $(K,D,\star)$  is separable. 

\end{proposition}
\section{\bf Some Properties of Probabilistic continuous and 1-Lipschitz map}\label{S2}
We recall from \cite{Ba} the notion of probabilistic Lipschitz maps and probabilistic continuous functions defined from a probabilistic metric space into $\Delta^+$. 
\begin{definition} (\cite{Ba}) Let $(G, D, \star)$ be a probabilistic metric space and let $f$ be a function $f : G \longrightarrow \Delta^+$. 

$(1)$ We say that $f$ is continuous at $z\in G$ if $f(z_n)\,{\xrightarrow {\textnormal{w}}}\, f(z)$, when $D(z_n,z)\,{\xrightarrow {\textnormal{w}}}\,\mathcal{H}_0$. We say that $f$ is continuous if $f$ is continuous at each point $z\in G$.

$(2)$ We say that $f$ is a probabilistic $1$-Lipschitz map if: $$\forall x, y \in G,\hspace{1mm} D(x,y)\star f(y)\leq f(x).$$
\end{definition}
\vskip5mm
We can also define $k$-Lipschitz maps for any nonegative real number $k\geq 0$ as the maps $f$ satisfying 
 $$\forall x, y \in G,\hspace{1mm} D_k(x,y)\star f(y)\leq f(x),$$
where, for all $x, y \in G$ and all $t\in \R$, $D_k(x,y)(t)=D(x,y)(\frac{t}{k})$ if $k>0$ and $D_0(x,y)(t)=\mathcal{H}_0(t)$ if $k=0$. For sake of simplicity, we shall only treat in this paper the case of $1$-Lipschitz maps, but our main result result could be easily extended to this more general setting.
\begin{example} \label{example.2} Let $(G,d)$ be a metric space. Assume that $\star$ is a triangle function on $\Delta^+$ satisfying $\mathcal{H}_a\star \mathcal{H}_b=\mathcal{H}_{a+b}$ for all $a, b\in \R^+$ (for example if $\star=\star_T$ where $T$ is a lef-continuous triangular norm). Let $(G, D, \star)$ be the probabilistic metric space defined with the probabilistic metric 
$$D(p,q)=\mathcal{H}_{d(p,q)}.$$
Let $L: (G,d) \longrightarrow \R^+$ be a real-valued map. Then, $L$ is a non-negative $1$-Lipschitz map if and only if $f : (G,D,\star) \longrightarrow \Delta^+$ defined for all $x\in G$ by
$$f(x):=\mathcal{H}_{L(x)}$$
is a probabilistic $1$-Lipschitz map. This example shows that the framework of probabilistic $1$-Lipschitz maps encompasses the classical determinist case.

\end{example}

By $C_\star(G,\Delta^+)$ we denote the space of all (probabilistic) continuous functions $f : G\longrightarrow \Delta^+$. We equip the space $C_\star(G,\Delta^+)$ with the following metric 
\begin{eqnarray*}
d_{\infty}(f,g):=\sup_{x\in G} d_L(f(x), g(x))
\end{eqnarray*}
where $d_L$ denotes the modified L\'evy distance on $\Delta^+$. 
\vskip5mm
By $Lip^1_\star(G,\Delta^+)$ we denote the space of all probabilistic $1$-Lipschitz maps   

$$Lip^1_\star(G,\Delta^+):=\lbrace f : G\longrightarrow \Delta^+/ D(x,y)\star f(y)\leq f(x); \forall x, y \in G\rbrace.$$
For all $x\in G$, by $\delta_x$ we denote the map 
\begin{eqnarray*}
\delta_x : G &\longrightarrow& \Delta^+\\
            y&\mapsto& D(y,x).
\end{eqnarray*}

We set $\mathcal{G}(G):=\lbrace \delta_x/ x\in G \rbrace$ and by $\delta$, we denote the operator  
\begin{eqnarray*}
\delta : G &\longrightarrow& \mathcal{G}(G)\subset Lip^1_\star(G,\Delta^+)\\
            x &\mapsto& \delta_x.
\end{eqnarray*}
\vskip5mm
The following proposition gives a canonical way to build probabilistic Lipschitz maps.
\begin{proposition} Let $(G, D, \star)$ be a probabilistic metric space such that $\star$ is sup-continuous. Let $f : (G,D,\star) \longrightarrow \Delta^+$ be any map and $A$ be any no-empty subset of $G$. Then, the map $\tilde{f}_A(x):=\sup_{y\in A} [ f(y) \star D(x,y)]$, for all $x\in G$ is a probabilistic $1$-Lipschitz map and we have $\tilde{f}_A(x)\geq f(x)$, for all $x\in A$.
\end{proposition}
\begin{proof} The proof is similar to the standard inf-convolution construction. The fact that $\tilde{f}_A(x)\geq f(x)$ for all $x\in A$ is immediate from the definition of $\tilde{f}_A$. Let us now prove that it is probabilistic $1$-Lipschiptz. Let $x$, $y\in G$. Then, for all $z\in A$, we have
	\begin{multline*}
		\tilde{f}_A(y) = \sup_{z\in A}\left[f(z)\star D(y,z)\right]\geq f(z)\star D(y,z)\\
		\geq f(z)\star \left(D(y,x)\star D(x,z)\right)=\left(f(z)\star D(x,z)\right)\star D(y,x).
	\end{multline*}
	We get the conclusion by taking the supremum with respect to $z\in A$.
	\end{proof}
%
%
\begin{proposition} (\cite[Proposition 3.6]{Ba}) \label{cont} Let $(G,D,\star)$ be a probabilistic metric space such that $\star$ is continuous. Then, every probabilistic $1$-Lipschitz map defined on $G$ is continuous. In other words, we  have that $Lip^1_\star(G,\Delta^+)\subset C_\star(G,\Delta^+)$.
\end{proposition}

\begin{proposition} \label{complete} Let $(G, D, \star)$ be a probabilistic metric space. Then, the space $(C_\star(G,\Delta^+), d_{\infty})$ is a complete metric space.
\end{proposition}
\begin{proof} Let $(f_n)$ be a Cauchy sequence in $(C_\star(G,\Delta^+), d_{\infty})$. In particular, for each $x\in G$, $(f_n(x))$ is Cauchy in $(\Delta^+, d_L)$ which is compact (Lemma \ref{DS1}). Thus, there exists a function $f : G \longrightarrow \Delta^+$ such that the sequence $(f_n)$ pointwise converges to $f$ on $G$ (with respect to the metric $d_L$). It is easy to see that in fact $(f_n)$ uniformly converges to $f$, since it is Cauchy sequence in $(C_\star(G,\Delta^+), d_{\infty})$. We need to prove that $f$ is continuous on $G$. Let $x\in G$ and $(x_k)$ be a sequence such that $d_L(D(x_k,x), \mathcal{H}_0)\longrightarrow 0$, when $k\longrightarrow +\infty$. 
For all $\varepsilon >0$, there exists $N_\varepsilon\in \N$ such that 
\begin{eqnarray}\label{eqq1}
n \geq N_\varepsilon \Longrightarrow d_{\infty} (f_n, f):=\sup_{x\in G} d_L(f_n(x),f(x))\leq \varepsilon
\end{eqnarray}
Using the continuity of $f_{N_\varepsilon}$, we have that there exists $\eta(\varepsilon)>0$ such that
\begin{eqnarray}\label{eqq2}
d_L(D(x_k,x),\mathcal{H}_0) \leq \eta(\varepsilon) \Longrightarrow d_L(f_{N_\varepsilon}(x_k),f_{N_\varepsilon}(x)) \leq \varepsilon
\end{eqnarray}
Using (\ref{eqq1}) and (\ref{eqq2}), we have that  
\begin{eqnarray*}
d_L(f(x_k),f(x))&\leq& d_L(f(x_k),f_{N_\varepsilon}(x_k))+d_L(f_{N_\varepsilon}(x_k),f_{N_\varepsilon}(x))+ d_L(f_{N_\varepsilon}(x),f(x))\\
                &\leq& 3\varepsilon
\end{eqnarray*}
This shows that $f$ is continuous on $G$. Finally, we proved that every Cauchy sequence $(f_n)$ uniformly converges to a continuous function $f$. In other words, the space $(C_\star(G,\Delta^+), d_{\infty})$ is complete.
\end{proof}
\section{\bf The main result: Probabilistic Arzela-Ascoli theorem}\label{S3}
The following theorem is the main result of the paper. Its proof will be given after some intermediate lemmas.
\begin{theorem} \label{AA} Let $(K, D, \star)$ be a probabilistic complete metric space such that $\star$ is continuous and sup-continuous. Then, the following assertions are equivalent. 

$(1)$ $(K, D, \star)$ is compact,

$(2)$ the metric space $(Lip^1_\star(K,\Delta^+),d_{\infty})$ is compact (equivalently, $Lip^1_\star(K,\Delta^+)$ is a compact subset of $(C_\star(K,\Delta^+), d_{\infty})$).
\end{theorem}
\begin{lemma} \label{equi} Let $(G, D, \star)$ be a probabilistic metric space such that $\star$ is continuous. Then, the set $Lip_\star^1(G,\Delta^+)$ is uniformly  equicontinuous. In other words:\\
$\forall \varepsilon>0, \exists \eta(\varepsilon) >0 : \forall f\in Lip_\star^1(G,\Delta^+), \forall x, y\in G;$
$$ d_L(D(x,y),\mathcal{H}_0) <\eta(\varepsilon) \Longrightarrow d_L(f(x),f(y))<\varepsilon.$$
\end{lemma}
\begin{proof} Since $(\Delta^+,d_L)$ is a compact metric space (see Lemma \ref{DS1}) and $\star$ is continuous, then $\star$ is uniformly continuous from $\Delta^+\times \Delta^+$ into $\Delta^+$. It follows that \\
$\noindent \forall \varepsilon>0, \exists \eta(\varepsilon) >0 : \forall F\in \Delta^+, \forall x, y\in G:$
$$d_L(D(x,y),\mathcal{H}_0) <\eta(\varepsilon) \Longrightarrow d_L(D(x,y)\star F,F) <\varepsilon.$$
In particular, we have for all $f\in Lip_\star^1(G,\Delta^+)$ and all $x$, $y\in G$,
\begin{eqnarray} \label{eq2}
\max [d_L(D(x,y)\star f(x),f(x)),d_L(D(x,y)\star f(y),f(y))]<\varepsilon.
\end{eqnarray}
We are going to prove that 
\begin{eqnarray*}
d_L(f(x),f(y)) \leq \max [d_L(D(x,y)\star f(x),f(x)),d_L(D(x,y)\star f(y),f(y))].
\end{eqnarray*}
Let $h_1, h_2 > 0$ be such that  $A(D(x,y)\star f(x), f(x),h_1)$, $A(f(x), D(x,y)\star f(x),h_1)$, $A(D(x,y)\star f(y), f(y),h_2)$ and $A(f(y), D(x,y)\star f(y),h_2)$ hold, which means that for all $t\in (0,h_1^{-1})$ and all $t'\in (0,h_2^{-1})$ we have:
\begin{eqnarray*}
0\leq D(x,y)\star f(x)(t) &\leq & f(x)(t+h_1) + h_1 \\
0\leq f(x)(t) &\leq & D(x,y)\star f(x)(t+h_1) + h_1 \\
0\leq D(x,y)\star f(y)(t') &\leq & f(y)(t'+h_2) + h_2\\
0\leq f(y)(t') &\leq & D(x,y)\star f(y)(t'+h_2) + h_2 
\end{eqnarray*}
From the second, the fourth inequalities and the fact that $f$ is $1$-Lipschitz, we get that for all $t\in (0,h_1^{-1})$ and all $t'\in (0,h_2^{-1})$ we have:
\begin{eqnarray*}
0\leq f(x)(t) \leq  f(y)(t+h_1) + h_1 \\
0\leq f(y)(t') \leq  f(x)(t'+h_2) + h_2 
\end{eqnarray*} 
It follows that for all $s\in (0,\min(h_1^{-1},h_2^{-1}))$ we have:
\begin{eqnarray*}
0\leq f(x)(s) \leq  f(y)(s+ \max(h_1,h_2)) + \max(h_1,h_2) \\
0\leq f(y)(s) \leq  f(x)(s+ \max(h_1,h_2)) + \max(h_1,h_2) 
\end{eqnarray*}
Thus, we have that $d_L(f(x),f(y))\leq \max(h_1,h_2)$ for all $h_1, h_2 > 0$ such that  $A(D(x,y)\star f(x), f(x),h_1)$, $A(f(x), D(x,y)\star f(x),h_1)$, $A(D(x,y)\star f(y), f(y),h_2)$ and $A(f(y), D(x,y)\star f(y),h_2)$ hold. This implies that 
\begin{eqnarray*}
d_L(f(x),f(y)) \leq \max (d_L(D(x,y)\star f(x),f(x)),d_L(D(x,y)\star f(y),f(y))).
\end{eqnarray*}
Using the above inequality and (\ref{eq2}), we get the conclusion.
\end{proof}
We recall the following useful  proposition (see \cite{Ba}).
\begin{proposition} (\cite[Proposition 3.5]{Ba}) \label{util}
Let $(F_n), (L_n), (K_n) \subset (\Delta^+,\star)$. Suppose that 

$(a)$ the triangle function $\star$ is continuous,

$(b)$ $F_n\,{\xrightarrow {\textnormal{w}}}\, F$, $L_n\,{\xrightarrow {\textnormal{w}}}\, L$ and $K_n\,{\xrightarrow {\textnormal{w}}}\, K$. 

$(c)$ for all $n\in \N$, $F_n \star L_n \leq K_n$,

then, $F\star L\leq K$.
\end{proposition}
\begin{lemma} \label{Lip} Let $(G, D, \star)$ be a probabilistic metric space. Let $(f_n)$ be a sequence of $1$-Lipschitz maps and $L$ be a subset of $G$. Suppose that there exists a function $f$ defined on $L$ such that $f_n(x) \,{\xrightarrow {\textnormal{w}}}\, f(x)$, when $n \longrightarrow +\infty$, for all $x\in L$. Then, $f$ is $1$-Lipschitz on $L$.
\end{lemma}
\begin{proof} Since $f_n$ is $1$-Lipschitz map for each $n \in \N$, then we have, for all $x, y \in L$ and for all $n\in \N$:
\begin{eqnarray*}
D(x,y)\star f_n(x) \leq f_n(y)
\end{eqnarray*}
Using Proposition \ref{util}, we get that for all $x, y \in L$
\begin{eqnarray*}
D(x,y)\star f(x) \leq f(y)
\end{eqnarray*}
In other words, $f$ is $1$-Lipschitz maps on $L$.
\end{proof}
\begin{lemma} \label{unif} Let $(K, D, \star)$ be a probabilistic compact metric space and $(f_n)$ be a sequence of $1$-Lipschitz maps. Suppose that there exists a $1$-Lipschitz map $f$ such that $d_L(f_n(x),f(x))\longrightarrow 0$, when $n\longrightarrow +\infty$ for all $x\in K$. Then, $(f_n)$ converges uniformly to $f$, that is, $d_{\infty}(f_n,f)\longrightarrow 0$, when $n \longrightarrow +\infty$.
\end{lemma}
\begin{proof} Let $\varepsilon >0$ and, using Lemma \ref{equi}, let $\eta(\varepsilon)$ be the uniform equicontinuity module of $Lip_\star^1(G,\Delta^+)$. Since $(K, D, \star)$ is compact, there exists a finite set $A$ such that 
$K=\cup_{a\in A} N_a(\eta(\varepsilon))$. 
Since $d_L(f_n(a),f(a))\longrightarrow 0$, when $n\longrightarrow +\infty$ for all $a\in A$. Then, for each $a\in A$, there exists $P_a\in \N$ such that 
\begin{eqnarray*}
n\geq P_a \Longrightarrow d_L(f_n(a),f(a)) & \leq & \varepsilon
\end{eqnarray*}
Since $A$ is finite, we have that
\begin{eqnarray*}
n\geq \max_{a\in A} P_a \Longrightarrow \sup_{a\in A} d_L(f_n(a),f(a)) & \leq & \varepsilon
\end{eqnarray*}
Thus, for all $x\in K=\cup_{a\in A} N_a(\eta(\varepsilon))$, there exists $a\in A$ such that $x\in  N_a(\eta(\varepsilon))$ and so we have that for all $n\geq \max_{a\in A} P_a$ :
\begin{eqnarray*}
d_L(f_n(x),f(x)) &\leq& d_L(f_n(x),f_n(a)) + d_L(f_n(a),f(a)) + d_L(f(a),f(x))\\
                                                         &\leq& 3\varepsilon.
\end{eqnarray*} 
In other words, 
\begin{eqnarray*}
n\geq \max_{a\in A} P_a \Longrightarrow d_{\infty}(f_n,f):=\sup_{x\in K} d_L(f_n(x),f(x)) &\leq& 3\varepsilon
\end{eqnarray*}
\end{proof}
%
We finally need the following result from  \cite{Ba}.
\begin{lemma} (see \cite[Theorem 3.7]{Ba}) \label{MC} Let $(G, D, \star)$ be a probabilistic metric space such that $\star$ is sup-continuous and let $A$ be a nonempty subset of $G$. Let $f: A\longrightarrow \Delta^+$ be a probabilistic $1$-Lipschitz map. Then, there exists a probabilistic $1$-Lipschitz map $\tilde{f}: G\longrightarrow \Delta^+$ such that $\tilde{f}_{|A}=f$.
\end{lemma}
\vskip5mm
Let us now give the proof of our main result.
\begin{proof}[Proof of Theorem \ref{AA}] $(1) \Longrightarrow (2)$ Suppose that $(K, D, \star)$ is compact. Let $(f_n)_n \subset Lip^1_\star(K,\Delta^+)$ be a sequence. We need to shows that there exists a subsequence of $(f_n)$ that converges uniformly to some $1$-Lipschitz function. Indeed, since $K$ is compact, it is separable, that is, there existes a sequence $(x_k)_k\subset K$ which is dense in $K$ for the probabilistic metric $D$. Let us denote $L:=\lbrace x_k: k\in \N\rbrace$. We know from Lemma \ref{DS1} and Lemma \ref{DS2}  that $(\Delta^+,\textnormal{w})$ is a compact metrizable space. Thus, by Tykhonov theorem, the space $(\Delta^+)^{\N}$ is a compact metrizable. Hence, the sequence $(f_n)$ has a subsequence $(f_{\varphi(n)})$ that converges pointwise on $L$ to some function $f$, necessarily $1$-Lipschitz map on $L$ by Lemma \ref{Lip}. By Lemma \ref{MC} and Proposition \ref{cont}, $f$ extend to a $1$-Lipschitz map $\tilde{f}$ on $K$, and this extension is unique since $L$ is dense. Let us prove now that $(f_{\varphi(n)})$ converges pointwise on $K$ to $\tilde{f}$. Indeed, let $x\in K$ and let us choose a subsequence from $L$ denoted also by $(x_k)$ that converges to $x$, that is $D(x_k,x)\,{\xrightarrow {\textnormal{w}}}\, \mathcal{H}_0$. Since $(f_n)$ is a sequence of $1$-Lipschitz functions, then we have,
\begin{eqnarray*} 
D(x_k,x)\star f_{\varphi(n)}(x_k) &\leq& f_{\varphi(n)}(x)\\
D(x_k,x)\star f_{\varphi(n)}(x) &\leq& f_{\varphi(n)}(x_k)
\end{eqnarray*}
Since $\Delta^+$ is compact, to show that $(f_{\varphi(n)}(x))$ is convergent sequence, it suffices to prove that any convergent subsequence of  $(f_{\varphi(n)}(x))$ converges to the same limit $\tilde{f}(x)$. Indeed, let $f_{\psi(n)}(x)$ be a subsequence of $f_{\varphi(n)}(x)$ that converges to some $g(x)\in \Delta^+$. Using the above inequalities and Proposition \ref{util}, we get that
\begin{eqnarray*} 
D(x_k,x)\star \tilde{f}(x_k) &\leq& g(x)\\
D(x_k,x)\star g(x) &\leq& \tilde{f}(x_k)
\end{eqnarray*}
Using the fact $D(x_k,x)\,{\xrightarrow {\textnormal{w}}}\, \mathcal{H}_0$, the continuity of $\tilde{f}$ on $K$ and Proposition \ref{util} in the above inequalities, we get that $g(x)=\tilde{f}(x)$. Finally, we proved that there exists a subsequence $(f_{\varphi(n)})$ of $(f_n)$ that converges pointwise on $K$ to some $1$-Lipschitz map $\tilde{f}$. That is, $f_{\varphi(n)}(x)\,{\xrightarrow {\textnormal{w}}}\, \tilde{f}(x)$ for all $x\in K$. Using Lemma \ref{unif}, we get that $(f_{\varphi(n)})$ converges uniformly to $\tilde{f}$ on $K$. 
\vskip5mm
$(2)\Longleftarrow (1)$ Suppose that $(Lip^1_\star(K,\Delta^+),d_{\infty})$ is compact. Let $(x_n)$ be a sequence of $K$. We need to prove that $(x_n)$ has a convergent subsequence. Indeed, consider the sequence $(\delta_{x_n})$ of $1$-Lipschitz maps, defined by $\delta_{x_n}: x\mapsto D(x_n,x)$ for each $n\in \N$. By assumption, there exists a subsequence $(\delta_{x_{\varphi(n)}})$ that converges uniformly to some $1$-Lipschitz map, in particular it is a Cauchy sequence. In other words, we have 
$$\lim_{p,q\longrightarrow +\infty} \sup_{x\in K} d_L(\delta_{x_{\varphi(p)}}(x),\delta_{x_{\varphi(q)}}(x))=0.$$
In particular we have 
$$\lim_{p,q\longrightarrow +\infty} d_L(\delta_{x_{\varphi(p)}}(x_{\varphi(q)}),\mathcal{H}_0)=0.$$
Equivalently,
$$\lim_{p,q\longrightarrow +\infty} d_L(D(x_{\varphi(p)},x_{\varphi(q)}),\mathcal{H}_0)=0.$$
This shows that the sequence $(x_{\varphi(n)})$ is Cauchy in $K$ (see Lemma \ref{DS2}). Thus, the sequence $(x_{\varphi(n)})$ converges to some point $x\in K$, since $K$ is complete. 
\end{proof}

\bibliographystyle{amsplain}

\end{document}